\documentclass[a4paper, 12pt]{amsart}
\usepackage[a4paper, left=4.3cm, right=4.3cm, top=4cm, bottom=3.5cm ]{geometry}
\usepackage{amsmath, amscd}
\usepackage{amssymb, color, hyperref}
\usepackage{calligra}
\usepackage{mathrsfs}

\def\doi#1{{\small\href{https://doi.org/#1}{\path{doi:#1}}}}
\def\arxiv#1{{\small\href{http://www.arxiv.org/abs/#1}{\path{arXiv:#1}}}}
\def\url#1{{\small\href{#1}{\path{#1}}}}

\theoremstyle{plain}
\newtheorem{theorem}{\bf Theorem}[section]
\newtheorem{proposition}[theorem]{\bf Proposition}
\newtheorem{lemma}[theorem]{\bf Lemma}
\newtheorem{corollary}[theorem]{\bf Corollary}

\newtheorem{open-problem}[theorem]{\bf Open Problem}

\theoremstyle{definition}
\newtheorem{example}[theorem]{\bf Example}

\newtheorem{definition}[theorem]{\bf Definition}
\newtheorem{remark}[theorem]{\bf Remark}

\numberwithin{equation}{section}

\begin{document}

\title{Topological insights into Monoids and Module systems}

\author{Carmelo Antonio Finocchiaro and Doniyor Yazdonov}

\address{Dipartimento di Matematica e Informatica, Universit\`a di Catania,  95125 Catania, Italy}
\email{cafinocchiaro@unict.it}

\address{Department of Mathematics and Scientific Computing, University of Graz,  8010 Graz, Austria}
\email{doniyor.yazdonov@uni-graz.at}

\subjclass[2020]{20M12, 20M14, 13A15}

\keywords{spectral spaces,  module systems, ideal systems, multiplicative ideal theory}

\thanks{The first author was partially supported by GNSAGA, by the research project PIACERI ``ACIVA - Anelli commutativi, loro
	ideali e varietà algebriche'' and by the research project PRIN ``Squarefree Gröbner degenerations, special varieties and related
	topics''. The second author was supported by FWF, Project Number DOC 183-N}

\begin{abstract}
Let $H$ be a monoid and let $G$ denote its quotient groupoid. 
We introduce the Riemann–Zariski space $\textup{Zar}(G| H)$ associated with $H$ and prove that it is a spectral space. 
We then show that, if $H$ is an $s$-Prüfer monoid, the space $\textup{Zar}(G| H)$ is homeomorphic to the prime $s$-spectrum of $H$ endowed with the Zariski topology.

Next, given a finitary ideal system $r$ on $H$, we prove that the space of $r$-ideals $\mathcal{I}_r(H)$, endowed with the Zariski topology, is spectral, and that the prime $r$-spectrum of $H$ is proconstructible in $\mathcal{I}_r(H)$.

We also introduce a new topology on the set $\mathcal{X}$ of all generalized $H$-module systems and study its properties. In particular, we show that $\mathcal{X}$ and its subspace $\mathcal{X}_{\textup{fin}}$ of finitary generalized $H$-module systems are spectral spaces. 

Finally, we provide a characterization of those subsets of the set of overmonoids of $H$ that are quasi-compact.
\end{abstract}

\maketitle

\section{Introduction}

Let $K$ be a field and let $R$ be a subring of $K$. 
We denote by $\textup{Zar}(K| R)$ the set of all valuation domains having $K$ as quotient field and containing $R$ as a subring. The first topological study of the space $\textup{Zar}(K| R)$, in the case where $R$ is the prime subring of $K$, is due to Zariski, who proved that this space is quasi-compact when endowed with what is now known as the Zariski topology (see~\cite{zar44,zar-sam75}).  Subsequently, several authors showed using a variety of techniques that if $K$ is the quotient field of $R$, then $\textup{Zar}(K| R)$ is a spectral space (see, for example ~\cite{dobb87,dobb86}).

The aim of this paper is to extend several fundamental concepts and results from ring theory to the setting of monoids and to investigate how the algebraic 
and the topological properties of module systems are related. Let $G$ be a groupoid and $H$ a submonoid of $G$. We introduce the Riemann–Zariski space $\textup{Zar}(G| H)$ associated with the monoid $H$ and equip it with a natural analogue of the Zariski topology. In Theorem~\ref{Zar}, we prove via the ultrafilter criterion for spectral spaces (Lemma~\ref{spectral}) that $\textup{Zar}(G|H)$ is spectral.
When $G$ is the quotient groupoid of $H$ and $H$ is an $s$-Prüfer monoid, we show in Theorem~\ref{prüfer} that $\textup{Zar}(G| H)$ is homeomorphic to the prime $s$-spectrum of $H$, endowed with the standard Zariski topology.
We further introduce a natural topology on the set $\mathcal{I}_r(H)$ of all $r$-ideals of $H$. In Theorem~\ref{pronconst}, we prove that $\mathcal{I}_r(H)$ is a spectral space and that the prime $r$-spectrum is proconstructible in $\mathcal{I}_r(H)$. 

 The second part of this paper is devoted to providing further insight into ideal/module systems: the genesis of this deep and elegant theory originates from the notion of star operation introduced by Krull in \cite{Krull-35}. They turned out to be a powerful tool to study problems regarding factorizations of ideals and, more generally, they provide a perfect abstract setting to study many classes of rings from the point of view of Multiplicative Ideal Theory. Star operations were generalized by Okabe and Matsuda in \cite{Ok-Ma94}, who introduced the more flexible notion of semistar operation on integral domains. Over the last few decades, semistar operations have become central in several questions regarding commutative algebra. Particularly fascinating is the connection between semistar operations and rings of rational functions like Kronecker function rings and Nagata rings. For a deeper insight into this circle of ideas, see \cite{fo-lo-forum,fo-lo-survey,fo-lo-jpaa-2009,fon-lop03,fon-jar-san04}. The investigation regarding the ideal theory of monoids led to the definition of ideal systems and module systems, which are the counterparts of star/semistar operations in monoid theory (see \cite{h-c98,halter-koch-2025}). Motivations for studying them come from various directions, one of them being that several properties of the arithmetical structure of ideals of integral domains can be derived from monoid theory. 

In \cite{Ca-Da14}, the authors introduce a new topology on the set of all semistar operations on an integral domain $R$. Inspired by this, in Section \ref{newtoponX} we define and study a Zariski topology on the set $\mathcal{X}$ of all generalized  $H$-module systems. In Theorem~\ref{main1}, we show that $\mathcal{X}$ is spectral, and Corollary~\ref{corrr} establishes that the subspace $\mathcal{X}_{\textup{fin}}$ of finitary generalized  $H$-module systems is also spectral.
 Module systems and generalized $H$-module systems for monoids play a role analogous to semistar operations in ring theory. Remark \ref{diff1} highlights both similarities and differences in the behavior of these two structures.

In Section~\ref{characterizationn}, we characterize the conditions under which a subspace of $\mathcal{R}(G|H)$, the set of all overmonoids of $H$, is quasi-compact (Theorem~\ref{main2}). This result further illustrates another difference between semistar operations and generalized $H$-module systems (see Remark~\ref{diff2}). The section concludes with two generalizations of results from \cite{Ca-Da14} (Propositions~\ref{prop1} and~\ref{prop2}).

This paper is the first step towards studying the arithmetic of monoids through topological methods. In particular, it is a contribution to Problem 25 of the survey \cite{alfred-kim-loper}. Furthermore, it lays the groundwork for a potential topological proof of \cite[Corollary 3.9]{doniyor26}.

\bigskip
\section{Preliminaries} Throughout the paper, by a monoid $H=(H, *)$ we mean a set $H\neq \emptyset$, together with an associative and  commutative law of composition $*:H\times H\longrightarrow H$, such that $H$ possesses an identity element and a zero element. We typically express $(H, *)$ using multiplication, written as $a * b=a \cdot b= ab$. The identity element is denoted by $1_H$ or just $1$, and the zero element by $0_H$ or just $0$. We denote by $H^\times$ the set of invertible elements of $H$, and we set $H^\bullet=H\setminus\{0\}$. A monoid $H$ is called \textit{cancellative} if $ab=ac$ implies $b=c$ for all $a\in H^\bullet$ and $b,c\in H$. $H$ is said to be a \textit{groupoid} if $H^\bullet$ is a group; equivalently, $H^\bullet=H^{\times}$. The \textit{quotient groupoid} $G$ of a cancellative monoid $H$ is defined by 
\[G=\{a/b: a\in H,  \hspace{0.2cm} b\in H^\bullet\}.\]

Let $X$ be a set. Denote by $\mathcal{P}(X)$  the collection of all subsets of $X$ and by $\mathcal{P}_\textup{fin}(X)$ the collection of all finite subsets of $X$. 

Unless otherwise stated, we shall always assume that $1_H\neq 0_H$ and that the monoid $H$ is cancellative.

\medskip
Throughout, a ring $R$ is understood to be a commutative ring with identity.

\medskip
Our notation follows Halter-Koch’s Ideal Systems \cite{h-c98}, where any terminology not defined here can be found.

\bigskip
\subsection{Ideal systems} Let $H$ be a monoid. An \textit{ideal system} on $H$ is a map
\[r:\left\{
\begin{array}{rcl}
     \mathcal{P}(H) & \longrightarrow & \mathcal{P}(H) \\
    X & \longmapsto & X_r 
\end{array}
\right.\]
such that the following conditions hold for all subsets $X, Y\subseteq H$ and all elements $c\in H$: 
\begin{enumerate}
    \item[\textbf{(Id1)}] $X\cup\{0\}\subseteq X_r.$
    \item[\textbf{(Id2)}] $X\subseteq Y_r$ implies $X_r\subseteq Y_r.$
    \item[\textbf{(Id3)}] $cX_r=(cX)_r$. 
    \item[\textbf{(Id4)}] $cH\subseteq \{c\}_r.$
\end{enumerate}

We define $r_s: \mathcal{P}(H)\longrightarrow \mathcal{P}(H)$ by $$X_{r_s}=\bigcup_{E\in \mathcal{P}_{\textup{fin}}(X)}E_r.$$
The ideal system $r$ is called \textit{finitary} if $r=r_s$, that is  $$X_r=X_{r_s}:=\bigcup_{E\in \mathcal{P}_{\textup{fin}}(X)}E_r$$ holds for all subsets $X\subseteq H$.

Let $r$ be an ideal system on $H$. A subset $I\subseteq H$ is called an \textit{$r$-ideal} of $H$ if there exists a subset $X\subseteq H$ such that $I=X_r$. In this case, we call $I$ the \textit{$r$-ideal generated by} $X$. We denote by $\mathcal{I}_r(H)$ the set of all $r$-ideals of $H$. A proper $r$-ideal $P\in \mathcal{I}_r(H)$ is called \textit{prime} if $H\setminus P$ is a multiplicatively closed subset of $H$. We denote by $r$-$\textup{spec}(H)$ the set of all prime $r$-ideals of $H$.

\begin{example}
    \leavevmode
    \begin{enumerate}
        \item Let $H$ be a monoid. For $X\subseteq H$, we define
\[X_s=\left\{
\begin{array}{rcl}
     \{0\}, & if & X=\emptyset, \\
    XH, & if & X\neq  \emptyset.
\end{array}
\right.\]
It is straightforward to verify that $s: \mathcal{P}(H) \longrightarrow \mathcal{P}(H)$ constitutes an ideal system on $H$, which we refer to as the \textit{$s$-system} of $H$.

\item Let $R$ be a ring. For $X\subseteq R$, we define 
$$X_d:=\{a_1x_1+\ldots+a_nx_n : n\in \mathbb{N}_0, x_1, \ldots, x_n\in X, a_1, \ldots, a_n\in R\}$$ the ring ideal generated by $X$. It can also be easily verified that $d: \mathcal{P}(R) \longrightarrow \mathcal{P}(R)$ constitutes an ideal system on $R$. We call it the \textit{$d$-system} of $R$. 
    \end{enumerate}
\end{example}

Let $G$ be a quotient groupoid of $H$. A subset $X\subseteq G$ is called \textit{$H$-fractional}, if $cX\subseteq H$ for some $c\in H^\bullet$. For an $H$-fractional subset $X\subseteq G$, if $c\in H^\bullet $ is such that $cX\subseteq H$, then we define $X_r=c^{-1}(cX)_r$. A subset $I\subseteq G$ is called a \textit{fractional $r$-ideal} if there exists an $H$-fractional subset $X\subseteq G$ such that $I=X_r$. A fractional $r$-ideal $I$ is said to be \textit{$r$-invertible} if there exists a fractional $r$-ideal $J$ such that $I\cdot_r J=H$.
Finally, $H$ is called an \textit{$r$-Prüfer monoid} if every non-zero $r$-finitely generated fractional $r$-ideal is $r$-invertible.

\subsection{Semistar operations}
Let $R$ be an integral domain and $K$ be the quotient field of $R$. We  denote  by $\overline{\mathbf{F}}(R)$ the set of all nonzero $R$-submodules of $K$. 

A \textit{semistar operation} on $R$ is a function \[\star:\left\{
\begin{array}{rcl}
     \overline{\mathbf{F}}(R) & \longrightarrow & \overline{\mathbf{F}}(R) \\
    F & \longmapsto & F^\star 
\end{array}
\right.\]
 satisfying the following axioms, for any nonzero element $k\in K$ and every $E,F\in\overline{\mathbf{F}}(R)$:
 \begin{enumerate}
     \item $(kF)^\star =kF^\star;$
     \item $E\subseteq F$ implies $E^\star \subseteq F^\star;$
     \item $F\subseteq F^\star;$
     \item $(F^\star)^\star=F^\star.$
 \end{enumerate}

 Semistar operations were introduced by A. Okabe and R. Matsuda in 1994 (see \cite{Ok-Ma94}).

 \begin{example}
\leavevmode
\begin{enumerate}

\item 
Let $L$ be an overring of $R$. Define the map
\[
\star_{\{L\}} \colon \overline{\mathbf{F}}(R) \longrightarrow \overline{\mathbf{F}}(R),
\qquad 
A \longmapsto A^{\star_{\{L\}}} := AL
\]
for all $A \in \overline{\mathbf{F}}(R)$. 
Then $\star_{\{L\}}$ is a semistar operation on $R$.

\item 
Let $\mathcal{S}$ be a nonempty collection of semistar operations on $R$. 
The infimum of $\mathcal{S}$, denoted by $\bigwedge(\mathcal{S})$, is the semistar operation on $R$ defined by
\[
F^{\bigwedge(\mathcal{S})}
=
\bigcap_{\star \in \mathcal{S}} F^{\star}
\qquad \text{for all } F \in \overline{\mathbf{F}}(R).
\]

\item 
Let $\Delta$ be a nonempty collection of overrings of $R$. 
By \emph{(2)}, the map
\[
\star_{\Delta}
:=
\bigwedge \bigl(\{\, \star_{\{B\}} \mid B \in \Delta \,\}\bigr)
\]
is a semistar operation on $R$. Equivalently, $\star_{\Delta}$ is given by
\[
F^{\star_{\Delta}}
=
\bigcap_{B \in \Delta} FB
\qquad \text{for all } F \in \overline{\mathbf{F}}(R).
\]

\end{enumerate}
\end{example}

 \bigskip
\subsection{Generalized $H$-module systems.}
Let $G$ be the quotient groupoid of $H$. A $\textit{generalized H-module system}$ on $G$ is a map 
\[r:\left\{
\begin{array}{rcl}
     \mathcal{P}(G) & \longrightarrow & \mathcal{P}(G) \\
    A & \longmapsto & A_r 
\end{array}
\right.\]
such that the following conditions are satisfied for all subsets $A,B$ of $G$ and all elements $c\in G$:
\begin{enumerate}
    \item[\textbf{(Id1)}] $A\cup\{0\}\subseteq A_r.$
    \item[\textbf{(M2)}] $A\subseteq B$ implies $A_r\subseteq B_r.$
    \item[\textbf{(Id3)}] $cA_r=(cA)_r$.
    \item[\textbf{(M4)}] $HA_r=A_r.$
\end{enumerate}

 Our definition of a generalized $H$-module system is more general than $H$-module systems defined in \cite{halter-koch-2025}, where Axiom $\textbf{M2}$ is replaced by the Axiom \[\textbf{(Id2)} \hspace{0.1cm}  A\subseteq B_r \quad \textup{implies} \quad A_r\subseteq B_r.\]

 Axiom \textbf{(Id2)} implies \textbf{(M2)}, but the converse
does not hold. To see this, let
\[
H=\{0\}\cup\{t^n:n\in\mathbb N_0\},
\qquad
G=\{0\}\cup\{t^n:n\in\mathbb Z\},
\]
where the powers of \(t\) are distinct and \(0\) is absorbing.
Then \(G\) is the quotient groupoid of \(H\). For every
\(A\subseteq G\), define
\begin{equation}\label{example16}
A_r=
\begin{cases}
G, & \text{if } |A\setminus\{0\}|\geq 2,\\
AH\cup\{0\}, & \text{if } |A\setminus\{0\}|\leq 1.
\end{cases}
\end{equation}

We verify that \(r\) is a generalized \(H\)-module system.
Axioms \textbf{(Id1)} and \textbf{(M2)} follow directly from
the definition. Multiplication by a nonzero element of \(G\)
preserves the number of nonzero elements of a subset and
commutes with the map \(A\mapsto AH\cup\{0\}\). Thus
\textbf{(Id3)} holds for every nonzero scalar. Since $\emptyset_r=\{0\}_r=\{0\},$
it also holds for the scalar \(0\). Finally, both possible
values, \(G\) and \(AH\cup\{0\}\), are \(H\)-stable, proving
\textbf{(M4)}.

However,
\[
\{1,t\}\subseteq\{1\}_r=H,
\qquad
\{1,t\}_r=G\not\subseteq H=\{1\}_r.
\]
Consequently, \(r\) does not satisfy \textbf{(Id2)}.
Moreover,
\[
(\{1\}_r)_r=H_r=G\neq H=\{1\}_r,
\]
so \(r\) is not idempotent.

Module systems provide an appropriate framework for extending semistar ideal theory to monoids. They were introduced by Franz Halter-Koch in 2001 (see \cite{h-c01}) and subsequently refined in \cite{h-c03} and \cite{h-c11}, in parallel with related developments in ring theory (see, for example, \cite{bag-fon04}, \cite{fon-lop03} and \cite{fon-jar-san04}). We now present further examples.

\medskip
\begin{example}\label{example2645}
\leavevmode
\begin{enumerate}
\item 
Let $\mathscr{F}$ be a nonempty family of subsets $M \subseteq G$ satisfying:
\begin{itemize}
    \item $0 \in M$,
    \item $HM = M$,
    \item $cM \in \mathscr{F}$ for all $c \in G$.
\end{itemize}
Define
\[
r_{\mathscr{F}}(A)
=
\bigcap \{\, M \in \mathscr{F} \mid A \subseteq M \,\}.
\]
Then $r_{\mathscr{F}}$ is a generalized  $H$-module system on $G$.

\item\label{examp2} 
If $r$ is a generalized $H$-module system on $G$, then its finitary closure
\[
r_{s}(A)
=
\bigcup_{F \in \mathcal{P}_{\mathrm{fin}}(A)} r(F),
\]
 is also a generalized  $H$-module system on $G$.
\end{enumerate}
\end{example}

\medskip
\subsection{Topology.}  Let $H$ be a monoid. For a subset $X\subseteq H$, we set 
\begin{center}
    $\mathbf{V}_r(X)=\{P\in r$-$\textup{spec}(H) : P \supseteq X\}$.
\end{center}
Then $\{\mathbf{V}_r(X): X\subseteq H\}$ forms the system of closed sets for a topology on  $r$-$\textup{spec}(H)$, we call it \textit{Zariski topology}. (Observe that $\mathbf{V}_r(H)=\emptyset$, $\mathbf{V}_r(\{0\}_r)=r\textup{-spec}(H), \hspace{0.2cm}\mathbf{V}_r(X)\cup \mathbf{V}_r(Y)=\mathbf{V}_r(X_r\cap Y_r)$ for all $X,Y\subseteq H$, and $\bigcap_{\lambda\in I}\mathbf{V}_r(X_\lambda)=\mathbf{V}_r((\bigcup_{\lambda\in I}X_\lambda)_r)$ for any family $\{X_\lambda\}_{\lambda\in I}$ of subsets of $H$.)  For $f\in H$, we set $\mathbf{D}_r(f)=r$-\textup{spec}$(H)\setminus\mathbf{V}_r(\{f\})$, called a \textit{principal open set}. The system of principal open sets $\{\mathbf{D}_r(f):f\in H\}$ forms a basis for the Zariski topology on $r$-$\textup{spec}(H)$.
 
Let $G$ be a groupoid, and $H\subseteq G$ be a submonoid. 
We define $\mathcal{R}(G|H)$ to be the set of the submonoids $S$ of $G$ such that $H\subseteq S$. Then $\mathcal{R}(G|H)$ has a natural topology whose subbasis open sets are the sets of the type $$U(x):=\{S\in \mathcal{R}(G|H): x\in S\},$$ for every $x\in G$. Now let $\textup{Zar}(G|H)$ be the set of all valuation submonoids $S$ of $G$ with $H\subseteq S$. $\textup{Zar}(G|H)$ is a subspace of $\mathcal{R}(G|H)$, and  the subspace topology on $\textup{Zar}(G|H)$, induced by the topology on $\mathcal{R}(G|H)$ defined above, is called the \textit{Zariski topology}. The set $\textup{Zar}(G|H)$, equipped with the Zariski topology, is called the \textit{Riemann-Zariski space of $G$ over $H$}.

Let $X$ be any set. A nonempty subset $\mathscr{F}\subseteq \mathcal{P}(X)$ with $\emptyset\notin \mathscr{F}$ is called a \textit{filter} on $X$ if the following properties are satisfied:
\begin{enumerate}
    \item $U\cap V\in \mathscr{F}$, for every $U,V\in \mathscr{F}$,
    \item if $U\in \mathscr{F}$ and $U\subseteq V\subseteq X$, then $V\in \mathscr{F}$.
\end{enumerate}
The collection of all filters on $X$, ordered by set inclusion,  forms a partially ordered set. Maximal elements of this poset are said to be \textit{ultrafilters} on $X$.

A topological space is called \textit{spectral} if it is homeomorphic to the prime $d$-spectrum of a ring, endowed with the Zariski topology. 

Let $X$ be a spectral space. The \textit{constructible topology} on $X$ is the topology whose basis is: $$\mathcal{B}:=\{U\cap (X\setminus V): U,V \subseteq X \textup{ quasi-compact open subsets}\}.$$ A \textit{proconstructible} subset of a spectral space is a set which is closed with respect to the constructible topology.

We end the section with the following two fundamental examples.

\medskip
\begin{example}
\leavevmode
\begin{enumerate}
    \item     Let $\Delta\subseteq \mathcal{R}(G|H)$ be a nonempty subset. Then the map $r_\Delta$ on $\mathcal{P}(G)$ defined by 
    \[
    r_\Delta: \mathcal{P}(G) \longrightarrow \mathcal{P}(G), \quad A \longmapsto \bigcap_{S \in \Delta} (SA\cup\{0\})
\] is a generalized  $H$-module system on $G$.

\item Let $\tau \subseteq \mathcal{X}$ be a nonempty subset of the set of all generalized $H$-module systems. The infimum of $\tau$, denoted by $\bigwedge(\tau)$, is the generalized  $H$-module system on $G$ defined by
\[
A_{\bigwedge(\tau)} 
= 
\bigcap_{r \in \tau} A_{r}
\hspace{0.2cm} \text{for all } A \subseteq G.
\]
\end{enumerate}

\end{example}

\bigskip
 
\section{Generalizations from rings to monoids}
 In this section, we generalize several results from ring theory to monoid theory. We begin by presenting the necessary concepts and lemmas.  
 
\medskip
\begin{lemma}[\cite{h-c98}, Corollary (i), p. 170]\label{surjectivite}
    For every prime ideal $\mathfrak{p}\in s$-$\textup{spec}(H)$, there exists a valuation monoid $V$ for $H$ such that  $H\setminus \mathfrak{p}=H\cap V^\times$. 
\end{lemma}

\medskip
For any local monoid $H$, we let $\mathfrak{m}_H$ denote the maximal $s$-ideal of $H$. Let $H_1, H_2$ be $s$-local monoids. We say that $H_2$ \textit{dominates} $H_1$, written $H_1\leq_d H_2$, if $H_1$
is a submonoid of $H_2$ and $H_1\cap \mathfrak{m}_{H_2}=\mathfrak{m}_{H_1}$.

\medskip
\begin{lemma}\label{domination}
    Let $G$ be a groupoid and let $$\mathcal{L}(G):=\{H: H \text{ is a local submonoid of } G\}.$$
     The following properties hold.
    \begin{enumerate}
        \item $\leq_d$ is a partial order on $\mathcal{L}(G)$.
       
        \item Valuation monoids of $G$ are maximal elements of $(\mathcal{L}(G), \leq_d)$. 
    \end{enumerate}
\end{lemma}
\begin{proof}

1. Obvious. 
   
2. Let $V$ be a valuation monoid of $G$. Note that $V \in \mathcal{L}(G)$. Suppose that there exists a local submonoid $H \in \mathcal{L}(G)$ such that $V \le_d H$ and $V \neq H$. Choose an element $x \in H \setminus V$. Since $V$ is a valuation monoid, it follows that $x^{-1} \in \mathfrak{m}_V$. Consequently, $x^{-1} \in \mathfrak{m}_H$. This yields a contradiction: because $1 \in H$, we must have $1 \notin \mathfrak{m}_H$, whereas the preceding argument implies $1 \in \mathfrak{m}_H$. Hence,  $V$ is a maximal element of $\mathcal{L}(G)$.
\end{proof}

\medskip
 We proceed by recalling the basic properties of ultrafilters.
 \medskip
\begin{lemma}[\cite{jech98} Section 7]\label{ultrafilterproperty}
    Let $X$ be a set and $\mathscr{U}$ be a collection of subsets of $X$. Then the following conditions are equivalent:
    \begin{enumerate}
      \item $\mathscr{U}$ is an ultrafilter on $X$.
        \item $\mathscr{U}$ is a filter on $X$ and, if $Y,Z\subseteq X$ satisfy $Y\cup Z\in \mathscr{U}$, then either $Y\in \mathscr{U}$ or $Z\in \mathscr{U}$.
        \item $\mathscr{U}$ is a filter on $X$ and, for every subset $Y\subseteq X$, either $Y\in \mathscr{U}$ or $X\setminus Y\in \mathscr{U}$.
        
    \end{enumerate}
\end{lemma}

\bigskip
Consider an arbitrary set $X$, and let $\mathcal{S} \subseteq \mathcal{P}(X)$ be a nonempty collection of subsets of $X$. Given an ultrafilter $\mathscr{U}$ on $X$ and any subset $Y \subseteq X$, we define:
$$
Y(\mathscr{U}) := Y_{\mathcal{S}}(\mathscr{U}) := \{ x \in X : \forall S \in \mathcal{S},  S \cap Y \in \mathscr{U} \iff x \in S \}.
$$
The set $Y_{\mathcal{S}}(\mathscr{U})$ is said to be \textit{the ultrafilter limit set of $Y$ with respect to $\mathscr{U}$}.

\medskip
\begin{example}
\leavevmode
\begin{enumerate}
    \item Let $X$ be any set and let $x\in X$. Then, the collection \[\mathscr{U}_x=\{Y\subseteq X: x\in Y\}\] is an ultrafilter on $X$, called \textit{the principal ultrafilter generated by} $x$.
    \item Let $H$ be a monoid, and set $X := s\textup{-spec}(H)$ and 
\[
\mathcal{S} := \{\, \mathbf{D}_s(f) \mid f \in H \,\}.
\]
Let $Y \subseteq X$, and let $\mathscr{U}$ be an ultrafilter on $X$ such that 
$X \setminus Y \notin \mathscr{U}$. Define
\[
\mathfrak{p}_{\mathscr{U}}
:=
\{\, f \in H \mid \mathbf{V}_s(f) \in \mathscr{U} \,\}.
\]
Then one readily verifies that $\mathfrak{p}_{\mathscr{U}} \in s\textup{-spec}(H)$ and that
\[
Y_{\mathcal{S}}(\mathscr{U}) = \{\mathfrak{p}_{\mathscr{U}}\}.
\]
\end{enumerate}

\end{example}

\medskip
Here we present the so-called ultrafilter criterion, which we frequently use to prove that a topological space is spectral.  

\medskip
\begin{lemma}[\cite{Ca14}, Corollary 3.3]\label{spectral}
    Let $X$ be a topological space. Then $X$ is a spectral space if and only if $X$ is  $T_0$ and has a subbasis $\mathcal{S}$ of open sets such that for any ultrafilter $\mathscr{U}$ on $X$, the ultrafilter limit set $X_\mathcal{S}(\mathscr{U})$ is nonempty.
\end{lemma}

\bigskip

\begin{theorem}\label{Zar}
Let $G$ be a groupoid and let $H$ be a submonoid of $G$. Then $\mathrm{Zar}(G| H)$
is a spectral space.
\end{theorem}

\begin{proof}
Let $M$ be a submonoid of $G$. For $x\in G$, denote by $M[x]$ the submonoid of $G$
generated by $M$ and $x$. The space $\mathrm{Zar}(G| H)$ admits a subbasis of open
sets of the form
 \begin{center}
        $B(x):=\textup{Zar}(G|H[x])=$\\
        $=\{S\in \mathcal{R}(G|H[x]):S \textup{ is a valuation submonoid of } G\}$
    \end{center} 
for all $x\in G$.

We first show that $\mathrm{Zar}(G| H)$ is $T_0$. Let $S,T\in \textup{Zar}(G| H)$
with $S\neq T$. Then there exists $x\in G$ such that either $x\in S\setminus T$ or
$x\in T\setminus S$. Without loss of generality, assume $x\in S\setminus T$. Then
$S\in B(x)$, while $T\notin B(x)$, proving the $T_0$ property.

Next, let $\mathscr{U}$ be an ultrafilter on $\mathrm{Zar}(G| H)$, and define
\[
H_{\mathscr{U}}
:=
\{x\in G : B(x)\in \mathscr{U}\}.
\]
We claim that $H_{\mathscr{U}}\in \mathrm{Zar}(G| H)$. Let $x,y\in H_{\mathscr{U}}$. Then $B(x)\cap B(y)\in \mathscr{U}$, and since
$B(x)\cap B(y)\subseteq B(xy)$, Lemma \ref{ultrafilterproperty} follows that $B(xy)\in \mathscr{U}$, hence
$xy\in H_{\mathscr{U}}$.

Now let $x\in G^{\bullet}$ and suppose that $x\notin H_{\mathscr{U}}$. Then
$B(x)\notin \mathscr{U}$, so $\mathrm{Zar}(G| H)\setminus B(x)\in \mathscr{U}$.
Since every element of $\mathrm{Zar}(G| H)$ is a valuation submonoid, we have
\[
\mathrm{Zar}(G| H)\setminus B(x)\subseteq B(x^{-1}),
\]
and thus, again by Lemma \ref{ultrafilterproperty}, we obtain  $B(x^{-1})\in \mathscr{U}$, which implies $x^{-1}\in H_{\mathscr{U}}$. For any $x\in H$, we obtain $B(x)=\textup{Zar}(G|H)\in \mathscr{U}$, and consequently, $H\subseteq H_{\mathscr{U}}$.

These observations show that $H_{\mathscr{U}}$ is a valuation submonoid containing
$H$, and therefore $H_{\mathscr{U}}\in \mathrm{Zar}(G| H)$. Moreover, by
construction, for every $x\in G$ we have
\[
H_{\mathscr{U}}\in B(x)
\quad\Longleftrightarrow\quad
B(x)\in \mathscr{U}.
\]
Hence
\[
H_{\mathscr{U}}\in \mathrm{Zar}(G|H)_{\mathcal{S}}(\mathscr{U}),
\hspace{0.4cm}
\text{ where } \hspace{0.4cm} \mathcal{S}=\{B(x):x\in G\}.
\]
Finally,  Lemma~\ref{spectral} follows that $\mathrm{Zar}(G| H)$ is a
spectral space.
\end{proof}

\bigskip
Let $H$ be a monoid with quotient groupoid $G$. The natural map 
\[\delta:\left\{
\begin{array}{rcl}
     \textup{Zar}(G|H) & \longrightarrow & s\textup{-spec}(H) \\
    S & \longmapsto & \mathfrak{m}_S\cap H 
\end{array}
\right.\]
where $\mathfrak{m}_S$ is the maximal $s$-ideal of $S$, is called \textit{the domination map of $H$}.

\medskip
\begin{theorem}\label{prüfer}
Let $H$ be a monoid with quotient groupoid $G$. Then the domination map
\[
\delta : \mathrm{Zar}(G| H)\longrightarrow s\text{-}\mathrm{spec}(H)
\]
is continuous and surjective. Moreover, if $H$ is an $s$-Prüfer monoid, then
$\delta$ is a homeomorphism.
\end{theorem}

\begin{proof}
Let $\mathfrak{p}$ be a prime $s$-ideal of $H$. By Lemma~\ref{surjectivite},
there exists a valuation monoid $V$ of $H$ such that
$
H\setminus \mathfrak{p}=H\cap V^{\times}.$
Equivalently,
\[
\delta(V)=H\cap \mathfrak{m}_V=\mathfrak{p},
\]
which proves that $\delta$ is surjective.

To prove continuity, we show that
\[
\delta^{-1}(\mathbf{D}_s(x))=B(x^{-1})
\qquad\text{for all } x\in H^\bullet.
\]
Indeed, let $S\in \mathrm{Zar}(G| H)$. Then
\[
\delta(S)=\mathfrak{m}_S\cap H\in \mathbf{D}_s(x)
\iff x^{-1}\in S
\iff S\in \mathrm{Zar}(G| H[x^{-1}]),
\]
and the claim follows. Hence $\delta$ is continuous.

Now assume that $H$ is an $s$-Prüfer monoid. Let $V,W\in \mathrm{Zar}(G| H)$
such that $\delta(V)=\delta(W)=\mathfrak{p}$. Then both $V$ and $W$ dominate
$H_{\mathfrak{p}}$. Since $H$ is $s$-Prüfer, the localization
$H_{\mathfrak{p}}$ is a valuation monoid; therefore, by
Lemma~\ref{domination}, it is maximal with respect to domination. Consequently,
\[
H_{\mathfrak{p}}=V=W,
\]
which shows that $\delta$ is injective.

It remains to show that $\delta$ is an open map. It suffices to prove that
$\delta(B(x))$ is open in $s$-$\mathrm{spec}(H)$ for all $x\in G$. We claim that
\[
\delta(B(x))
=
s\text{-}\mathrm{spec}(H)\setminus \mathbf{V}_s((H:x)),
\]
where $(H:x)=\{h\in H: hx\in H\}$.

Let $\mathfrak{p}\in s$-$\mathrm{spec}(H)$. If $\mathfrak{p}\in \delta(B(x))$, then
$H_{\mathfrak{p}}\in B(x)$, since $\delta(H_{\mathfrak{p}})=\mathfrak{p}$. Write
$x=h/s$ with $h\in H$ and $s\in H\setminus \mathfrak{p}$. Then
$s\in (H:x)\setminus \mathfrak{p}$, so $(H:x)\nsubseteq \mathfrak{p}$, that is,
$\mathfrak{p}\notin \mathbf{V}_s((H:x))$.

Conversely, if $(H:x)\nsubseteq \mathfrak{p}$, choose
$s\in (H:x)\setminus \mathfrak{p}$. Then $s\in H_{\mathfrak{p}}^{\times}$ and
$sx\in H\subseteq H_{\mathfrak{p}}$, which implies $x\in H_{\mathfrak{p}}$.
Therefore $H_{\mathfrak{p}}\in B(x)$ and
$\mathfrak{p}\in \delta(B(x))$.
This proves the claimed equality and shows that $\delta$ is an open map. 
\end{proof}

\medskip
\begin{lemma}[\cite{dickman19} Theorem 2.1.3]\label{proconstr} Let $X$ be a spectral space and $Y\subseteq X$ be a subset. If $Y$ is proconstructible in $X$, then $Y$ is a spectral space. 
    \end{lemma}

\medskip
\begin{definition}
    We endow the set of all $r$-ideals $\mathcal{I}_r(H)$ with a topology whose subbasic open sets are of the form $$\mathbf{U}_r(x):=\{I\in \mathcal{I}_r(H): x\notin I\}$$ for all $x\in H$. This topology is called \textit{the Zariski topology} on $\mathcal{I}_r(H)$.
\end{definition}
The induced topology on $r\textup{-spec}(H)\subset \mathcal{I}_r(H)$ is  the standard Zariski topology on $r\textup{-spec}(H)$.

\medskip
\begin{theorem}\label{pronconst}
Let $H$ be a monoid and let $r$ be a finitary ideal system on $H$. 
Then $\mathcal{I}_r(H)$ is a spectral space and $r\textup{-spec}(H)$ is proconstructible in $\mathcal{I}_r(H)$. 
In particular, $r\textup{-spec}(H)$ is a spectral space.
\end{theorem}

\begin{proof}
We first show that $\mathcal{I}_r(H)$ satisfies the $T_0$ separation axiom. 
Let $I,J \in \mathcal{I}_r(H)$ with $I \neq J$. Then there exists $x \in H$ such that either $x \in I \setminus J$ or $x \in J \setminus I$. 
Without loss of generality, assume $x \in I \setminus J$. Then
$I \notin \mathbf{U}_r(x)$ and $J \in \mathbf{U}_r(x),$
  so $I$ and $J$ are topologically distinguishable. Hence $\mathcal{I}_r(H)$ is $T_0$.

Let $\mathcal{S} := \{\mathbf{U}_r(x) \mid x \in H\}$, which forms a subbasis for the Zariski topology on $\mathcal{I}_r(H)$. 
Let $\mathscr{U}$ be an ultrafilter on $\mathcal{I}_r(H)$ and define
\[
I_{\mathscr{U}}
:=
\{\, y \in H \mid \mathbf{U}_r(y) \notin \mathscr{U} \,\}.
\]

We claim that $I_{\mathscr{U}}$ is an $r$-ideal and that 
$I_{\mathscr{U}} \in ( \mathcal{I}_r(H) )_{\mathcal{S}}(\mathscr{U})$. To show that $I_{\mathscr{U}}$ is an $r$-ideal, it suffices to prove that
$I_{\mathscr{U}} = (I_{\mathscr{U}})_r$. 
The inclusion $I_{\mathscr{U}} \subseteq (I_{\mathscr{U}})_r$ is clear. 
For the reverse inclusion, let $x \in (I_{\mathscr{U}})_r$. 
Since $r$ is finitary,
\[
(I_{\mathscr{U}})_r 
=
\bigcup_{E \in \mathcal{P}_{\mathrm{fin}}(I_{\mathscr{U}})} E_r.
\]
Thus there exists $E = \{y_1,\dots,y_n\} \subseteq I_{\mathscr{U}}$ such that $x \in E_r$.

For each $i$, we have $\mathbf{U}_r(y_i) \notin \mathscr{U}$, and hence
\[
\mathcal{I}_r(H) \setminus \mathbf{U}_r(y_i) \in \mathscr{U}.
\]
By Lemma~\ref{ultrafilterproperty},
\[
\bigcap_{i=1}^n 
\bigl( \mathcal{I}_r(H) \setminus \mathbf{U}_r(y_i) \bigr)
\in \mathscr{U}.
\]
Let $J$ be an $r$-ideal in this intersection. Then $y_i \in J$ for all $i$, so $E \subseteq J$. 
Since $r$ is monotone and idempotent \cite[Proposition, p.~16]{h-c98}, we obtain
\[
E_r \subseteq J_r = J.
\]
In particular, $x \in J$, so $J \in \mathcal{I}_r(H) \setminus \mathbf{U}_r(x)$. Therefore,
\[
\bigcap_{i=1}^n 
\bigl( \mathcal{I}_r(H) \setminus \mathbf{U}_r(y_i) \bigr)
\subseteq 
\mathcal{I}_r(H) \setminus \mathbf{U}_r(x).
\]
Applying Lemma~\ref{ultrafilterproperty} again, we conclude that
$\mathbf{U}_r(x) \notin \mathscr{U}$, that is, $x \in I_{\mathscr{U}}$. 
Hence $(I_{\mathscr{U}})_r \subseteq I_{\mathscr{U}}$, and thus $I_{\mathscr{U}}$ is an $r$-ideal.

The inclusion 
$I_{\mathscr{U}} \in ( \mathcal{I}_r(H) )_{\mathcal{S}}(\mathscr{U})$ 
follows directly from the definition of $I_{\mathscr{U}}$. 
Therefore, by Lemma~\ref{spectral}, the space $\mathcal{I}_r(H)$ endowed with the Zariski topology is spectral.

Next, let $a,b \in H$ and set
\[
O_{a,b}
:=
\mathbf{U}_r(a) \cap \mathbf{U}_r(b)
\cap
\bigl( \mathcal{I}_r(H) \setminus \mathbf{U}_r(ab) \bigr).
\]
By \cite[Corollary~1.2]{fin-fon-dar16}, each $\mathbf{U}_r(x)$ is quasi-compact  in $\mathcal{I}_r(H)$. 
Hence $O_{a,b}$ is open in the constructible topology.

Moreover, since a proper $r$-ideal $I$ is not prime if and only if there exist $a,b\in H$ such that $ab\in I$ and $a,b\notin I$, we obtain
\[
\mathcal{I}_r(H) \setminus r\textup{-spec}(H)
=
(\mathcal{I}_r(H)\setminus U_r(1))\cup\bigcup_{a,b \in H} O_{a,b},
\]
which shows that $r\textup{-spec}(H)$ is proconstructible in $\mathcal{I}_r(H)$. 
Finally, Lemma~\ref{proconstr} implies that $r\textup{-spec}(H)$ is a spectral space.
\end{proof}

\medskip
\begin{remark}
The ring-theoretic versions of Theorem~\ref{Zar} and Theorem~\ref{prüfer} 
have already been established. 
See \cite{fin-fon-lop13} for a constructive proof of the ring case of Theorem~\ref{Zar}, 
and \cite[Proposition 2.2]{dobb87} for the ring case of Theorem~\ref{prüfer}. 
Moreover, in \cite{ju12}, J.~Juett proved that $r\textup{-spec}(H)$ is a spectral space 
by applying Hochster's characterization (Lemma~\ref{Hochster}).
\end{remark}

\bigskip
\section{The Zariski topology on the set of all generalized $H$-module systems}\label{newtoponX}

In this section, we introduce a new topology on the set of all the generalized $H$-module systems, inspired by \cite{Ca-Da14}, which we call the \textit{Zariski topology}. We show that this topological space is spectral using the ultrafilter criterion. Moreover, we also show that the set of all finitary generalized  $H$-module systems with the induced topology is a spectral space.

\medskip
\begin{definition}\label{defnzartop}
    Let $\mathcal{X}$ be the set of all the generalized  $H$-module systems on $G$. The \textit{Zariski topology} on $\mathcal{X}$ is the topology for which a subbasis
of open sets is the collection of all sets of the form $$U_S:=\{r\in \mathcal{X}: 1\in S_r\},$$ as $S$ ranges over all nonzero subsets of $G$.
\end{definition}
We denote by $\mathcal{X}_\textup{fin}$  the set of all finitary generalized  $H$-module systems. It is endowed with the subspace topology.

\medskip
\begin{theorem}\label{main1}
Let $H$ be a monoid and let $G$ be its quotient groupoid. Then the space
$\mathcal{X}$, endowed with the Zariski topology, is a spectral space.
\end{theorem}

\begin{proof}
Let $\mathscr{U}$ be an ultrafilter on $\mathcal{X}$, and let
\[
\mathcal{S}:=\{U_S : S\subseteq G\}
\]
be the canonical subbasis of the Zariski topology. By Lemma~\ref{spectral}, it
suffices to show that $\mathcal{X}$ is $T_0$ and that the set
\[
\mathcal{X}(\mathscr{U})
:=
\{r\in \mathcal{X} : \forall U_S\in \mathcal{S},\ r\in U_S \iff U_S\in \mathscr{U}\}
\]
is nonempty.

We first show that $\mathcal{X}$ is $T_0$. Let $r_1\neq r_2$ be elements of
$\mathcal{X}$. Then there exist a subset $S\subseteq G$ and an element
$g\in G^{\bullet}$ such that $g\in S_{r_1}$ and $g\notin S_{r_2}$. By axiom~\textbf{(Id3)},
for any $r\in \mathcal{X}$ we have
\[
g\in S_r \iff 1\in (g^{-1}S)_r.
\]
Hence $r_1\in U_{g^{-1}S}$ while $r_2\notin U_{g^{-1}S}$, which proves that
$\mathcal{X}$ is $T_0$.

For each subset $S\subseteq G$ and element $g\in G$, set
\[
U_{S,g}:=\{r\in \mathcal{X}: g\in S_r\}.
\]
Define a map
\[
r_{\mathscr{U}}:\mathcal{P}(G)\longrightarrow\mathcal{P}(G),
\qquad
S\longmapsto \{g\in G : U_{S,g}\in \mathscr{U}\}.
\]
We claim that $r_{\mathscr{U}}$ is a generalized  $H$-module system on $G$, and that
$r_{\mathscr{U}}\in \mathcal{X}(\mathscr{U})$.

\medskip
\noindent
\textbf{(Id1)} Let $A\subseteq G$ and $a\in A\cup\{0\}$. Since
$A\cup\{0\}\subseteq A_r$ for all $r\in\mathcal{X}$, we have $a\in A_r$ for all
$r$. Thus $U_{A,a}=\mathcal{X}\in \mathscr{U}$, and hence
$a\in A_{r_{\mathscr{U}}}$. Therefore,
$A\cup\{0\}\subseteq A_{r_{\mathscr{U}}}$.

\medskip
\noindent
\textbf{(M2)} Let $A\subseteq B\subseteq G$ and $g\in A_{r_{\mathscr{U}}}$. Then
$U_{A,g}\in \mathscr{U}$. Since $A_r\subseteq B_r$ for all $r\in\mathcal{X}$, we
have $U_{A,g}\subseteq U_{B,g}$. By upward closure of $\mathscr{U}$,
$U_{B,g}\in \mathscr{U}$, and hence $g\in B_{r_{\mathscr{U}}}$. Thus
$A_{r_{\mathscr{U}}}\subseteq B_{r_{\mathscr{U}}}$.

\medskip
\noindent
\textbf{(Id3)} Let $A\subseteq G$ and $c,g\in G$. For $c=0$, the required equality follows from
    $\emptyset_{r_{\mathscr U}}=\{0\}_{r_{\mathscr U}}=\{0\}$.
    Now let $c\in G^\bullet$ and $g\in G$. By definition,
\[
g\in cA_{r_{\mathscr{U}}}
\iff U_{A,c^{-1}g}\in \mathscr{U}.
\]
On the other hand,
\[
U_{A,c^{-1}g}
=
\{r\in \mathcal{X}: c^{-1}g\in A_r\}
=
\{r\in \mathcal{X}: g\in (cA)_r\}
=
U_{cA,g}.
\]
Thus $U_{cA,g}\in \mathscr{U}$ if and only if $g\in (cA)_{r_{\mathscr{U}}}$,
showing that
\[
cA_{r_{\mathscr{U}}}=(cA)_{r_{\mathscr{U}}}.
\]

\medskip
\noindent
\textbf{(M4)} Let $g\in HA_{r_{\mathscr{U}}}$. Then $g=ha$ for some $h\in H$ and
$a\in A_{r_{\mathscr{U}}}$. Since $a\in A_{r_{\mathscr{U}}}$, we have
$U_{A,a}\in \mathscr{U}$. For every $r\in\mathcal{X}$, the equality
$HA_r=A_r$ holds, so $a\in A_r$ implies $ha\in A_r$. Thus
$U_{A,a}\subseteq U_{A,g}$, and by upward closure,
$U_{A,g}\in \mathscr{U}$. Hence $g\in A_{r_{\mathscr{U}}}$. The reverse inclusion follows from $1\in H$.

\medskip
These verifications show that $r_{\mathscr{U}}$ is a generalized  $H$-module system on $G$.
Finally, by construction, for every subbasic open set $U_S$ we have
\[
r_{\mathscr{U}}\in U_S \iff U_S\in \mathscr{U},
\]
which implies that $r_{\mathscr{U}}\in \mathcal{X}(\mathscr{U})$.

Therefore, $\mathcal{X}(\mathscr{U})$ is nonempty, and by
Lemma~\ref{spectral}, the space $\mathcal{X}$ is spectral.
\end{proof}

\bigskip
\begin{remark}
A generalized  $H$-module system $r$ on $G$ satisfies Axiom \textbf{Id2} if and only if it is idempotent, that is,
$(A_r)_r = A_r$  for all  $A \subseteq G.$
We claim that, in general, $r_{\mathscr{U}}$ need not be idempotent. Let $r$ be the generalized $H$-module system defined in ~\eqref{example16}, and let $\mathscr{U}$ be the principal ultrafilter generated by $r$. Then, for every $S \subseteq G$, we have
\[
S_{r_{\mathscr{U}}}
=
\{\, g \in G :U_{S,g} \in \mathscr{U} \,\}
=
\{\, g \in G : r \in U_{S,g} \,\}
=
S_r.
\]
Hence $r_{\mathscr{U}} = r$. Since
\[
(\{1\}_r)_r = G \neq \{1\}_r=H,
\]
it follows that $r$ is not idempotent. Therefore, $r_{\mathscr{U}}$ does not satisfy Axiom \textbf{Id2}.
\end{remark}
\medskip

Let $\mathcal{M}_{\textup{fin}}$ denote the set of all finitary
$H$-module systems, endowed with the topology induced by
Definition~\ref{defnzartop}. For $S\subseteq G$ and $g\in G$, set
$$
\mathcal V_{S,g}:=
U_{S,g}\cap\mathcal{M}_{\textup{fin}}.
$$
Since every system in $\mathcal{M}_{\textup{fin}}$ is finitary,
$$
\mathcal V_{S,g}
=
\bigcup_{F\in\mathcal P_{\textup{fin}}(S)}
\mathcal V_{F,g}.
$$
Thus the sets $\mathcal V_{F,g}$, with $F\subseteq G$ finite
and $g\in G$, form a subbasis for this topology.

Let $\mathscr U$ be an ultrafilter on
$\mathcal{M}_{\textup{fin}}$. Define
$$
r_{\mathscr U}:
\left\{
\begin{array}{rcl}
\mathcal P(G) & \longrightarrow & \mathcal P(G)\\[2mm]
S & \longmapsto &
\left\{
g\in G:
\exists F\in\mathcal P_{\textup{fin}}(S)
\text{ such that }\mathcal V_{F,g}\in\mathscr U
\right\}.
\end{array}
\right.
$$
The verifications of \textbf{(Id1)}, \textbf{(M2)},
\textbf{(Id3)}, and \textbf{(M4)} are analogous to those in
Theorem~\ref{main1}, using finite witnesses. 

We verify that $r_{\mathscr U}$ is idempotent. Suppose that
$g\in(S_{r_{\mathscr U}})_{r_{\mathscr U}}$. Choose a finite
subset $F\subseteq S_{r_{\mathscr U}}$ such that
$\mathcal V_{F,g}\in\mathscr U$. For each $f\in F$, choose
a finite subset $E_f\subseteq S$ such that
$\mathcal V_{E_f,f}\in\mathscr U$, and put
$E:=\bigcup_{f\in F}E_f$. Every system $r$ belonging to
$$
\mathcal V_{F,g}\cap
\bigcap_{f\in F}\mathcal V_{E_f,f}
$$
satisfies $F\subseteq E_r$ and $g\in F_r$. By
\textbf{(Id2)}, $g\in E_r$. Hence
$\mathcal V_{E,g}\in\mathscr U$, so
$g\in S_{r_{\mathscr U}}$. The reverse inclusion follows
from \textbf{(Id1)}. Therefore
$r_{\mathscr U}\in\mathcal{M}_{\textup{fin}}$.

For every finite $F\subseteq G$ and every $g\in G$,
monotonicity gives
$$
r_{\mathscr U}\in\mathcal V_{F,g}
\quad\Longleftrightarrow\quad
g\in F_{r_{\mathscr U}}
\quad\Longleftrightarrow\quad
\mathcal V_{F,g}\in\mathscr U.
$$
Consequently, every ultrafilter has a limit relative to the
finite-input subbasis. Since $\mathcal{M}_{\textup{fin}}$
inherits the $T_0$ property from $\mathcal X$, the
ultrafilter criterion shows that
$\mathcal{M}_{\textup{fin}}$ is spectral.

\medskip
We proceed by presenting lemmas and propositions to show that
$\mathcal{X}_{\textup{fin}}$ is also a spectral space.

\medskip
\begin{lemma}[\cite{hoch69} Theorem 1 and Proposition 4]\label{hausd}
    Let $X$ be a spectral space. The constructible topology is Hausdorff, totally disconnected, and quasi-compact.
\end{lemma}

\medskip
\begin{lemma}[The Stacks Project, Tag 08ZP]\label{alex} Let $(X, \tau)$ be a topological space, and $\mathcal{S}$ be a subbasis of $\tau$. Then $X$ is quasi-compact if and only if every cover by elements from $\mathcal{S}$ has a finite subcover.   
    \end{lemma}

\medskip
\begin{proposition}\label{quasi-compactt}
Let $A\subseteq G$ and $x\in G$. Then the subspace
\[
U_{A,x}
=
\{r\in \mathcal{X} : x\in A_r\}
\]
is quasi-compact in $\mathcal{X}$.
\end{proposition}

\begin{proof}
   
By definition, the family of sets
\[
\mathcal{S}
=
\bigl\{U_{A,x} : A\subseteq G,\ x\in G\bigr\}
\]
forms a subbasis for the topology on $\mathcal{X}$. Hence, by
Lemma~\ref{alex}, it suffices to show that every open cover of $U_{A,x}$ by
subbasic open sets admits a finite subcover.
If $U_{A,x}=\emptyset$, the conclusion is immediate.
    Hence assume that $U_{A,x}\neq\emptyset$ and that
\begin{equation}\label{eq1}
U_{A,x}
\subseteq
\bigcup_{\lambda\in I} U_{A_\lambda,x_\lambda}
\end{equation}
is a cover by subbasic open sets. By construction,
$\bigwedge(U_{A,x})$ is the smallest generalized $H$-module system (with respect to
pointwise inclusion) such that
$x\in A_{\bigwedge(U_{A,x})}$; equivalently,
$\bigwedge(U_{A,x})\in U_{A,x}$. Therefore, the cover \eqref{eq1} implies that
there exists $\lambda_0\in I$ such that
\begin{equation}\label{eq2}
\bigwedge(U_{A,x})\in U_{A_{\lambda_0},x_{\lambda_0}}.
\end{equation}
We claim that
\[
U_{A,x}\subseteq U_{A_{\lambda_0},x_{\lambda_0}}.
\]
Let $r\in U_{A,x}$. From \eqref{eq2} we obtain
\[
x_{\lambda_0}\in (A_{\lambda_0})_{\bigwedge(U_{A,x})}
\subseteq (A_{\lambda_0})_{r},
\]
where the inclusion follows from the minimality of
$\bigwedge(U_{A,x})$. Hence $r\in U_{A_{\lambda_0},x_{\lambda_0}}$, proving the
claim.

Thus every subbasic open cover of $U_{A,x}$ admits a finite subcover. Consequently, $U_{A,x}$ is quasi-compact.
\end{proof}

The quasi-compactness of $U_{A,x}$ can also be deduced from \cite[Corollary 1.2]{fin-fon-dar16}, which asserts that if a subbasis of a spectral space makes the ultrafilter criterion work, then its elements are quasi-compact open sets. From Theorem \ref{main1}, we know that the space $\mathcal{X}$ is  spectral and the canonical subbasis $$\mathcal{S}=\{U_{A,x}:A\subseteq G, x\in G\}$$ makes the ultrafilter criterion work. Therefore, each set $U_{A,x}$ is quasi-compact.

\medskip
\begin{proposition}\label{propwje}
The map
\[
\Phi:\mathcal X\longrightarrow\mathcal X_{\mathrm{fin}},
\qquad
r\longmapsto r_s,
\]
is a continuous retraction when both spaces are endowed
with their Zariski topologies.
\end{proposition}

\begin{proof}
By Example~\ref{example2645}(\ref{examp2}), $r_s$ is a generalized $H$-module
system. For every finite subset $F\subseteq G$,
monotonicity gives
\[
F_{r_s}
=
\bigcup_{E\in\mathcal P_{\mathrm{fin}}(F)}E_r
=
F_r.
\]
Consequently, $(r_s)_s=r_s$. Thus $\Phi$ takes values
in $\mathcal X_{\mathrm{fin}}$ and restricts to the
identity on $\mathcal X_{\mathrm{fin}}$. Let $A\subseteq G$ and $x\in G$. By the definition
of finitarization,
\[
\begin{aligned}
\Phi^{-1}\bigl(U_{A,x}\cap\mathcal X_{\mathrm{fin}}\bigr)
&=\{r\in\mathcal X:x\in A_{r_s}\}\\
&=\bigcup_{F\in\mathcal P_{\mathrm{fin}}(A)}U_{F,x}.
\end{aligned}
\]
This is open in $\mathcal X$. Since the sets
$U_{A,x}\cap\mathcal X_{\mathrm{fin}}$ form a subbasis
for the subspace topology, $\Phi$ is continuous.
\end{proof}

\medskip
\begin{corollary}\label{corrr}
The space $\mathcal X_{\mathrm{fin}}$, endowed with the
subspace Zariski topology, is spectral.
\end{corollary}

\begin{proof} For every
$A\subseteq G$ and $x\in G$, finitarity gives
\[
U_{A,x}\cap X_{\mathrm{fin}}
=
\bigcup_{F\in\mathcal P_{\mathrm{fin}}(A)}
\bigl(U_{F,x}\cap X_{\mathrm{fin}}\bigr).
\]
Hence
\[
\mathcal S_{\mathrm{fin}}
:=
\left\{
U_{F,x}\cap X_{\mathrm{fin}}:
F\in\mathcal P_{\mathrm{fin}}(G),\ x\in G
\right\}
\]
is a subbasis for the topology on $X_{\mathrm{fin}}$.
Moreover, $X_{\mathrm{fin}}$ is $T_0$, since it is a subspace of
the $T_0$ space $\mathcal X$.

Let $\mathscr U$ be an ultrafilter on $X_{\mathrm{fin}}$. Define
\[
\rho:\mathcal P(G)\longrightarrow\mathcal P(G),
\qquad
A_\rho
:=
\{x\in G:U_{A,x}\cap X_{\mathrm{fin}}\in\mathscr U\}.
\]
The same verification as in the proof of Theorem~\ref{main1}
shows that $\rho$ is a generalized $H$-module system.
For the zero scalar in \textbf{(Id3)}, note explicitly
that $\emptyset_\rho=\{0\}_\rho=\{0\},$
because these equalities hold for every system in $X_{\mathrm{fin}}$. Let $r_{\mathscr U}:=\rho_s$. Then
$r_{\mathscr U}\in X_{\mathrm{fin}}$. Explicitly,
\[
A_{r_{\mathscr U}}
=
\left\{
x\in G:
\text{there exists }F\in\mathcal P_{\mathrm{fin}}(A)
\text{ such that }U_{F,x}\cap X_{\mathrm{fin}}\in\mathscr U
\right\}.
\]
For every finite subset $F\subseteq G$, monotonicity
of $\rho$ yields
\[
F_{r_{\mathscr U}}
=
\bigcup_{E\in\mathcal P_{\mathrm{fin}}(F)}E_\rho
=
F_\rho.
\]
Therefore,
\[
r_{\mathscr U}\in U_{F,x}\cap X_{\mathrm{fin}}
\quad\Longleftrightarrow\quad
x\in F_\rho
\quad\Longleftrightarrow\quad
U_{F,x}\cap X_{\mathrm{fin}}\in\mathscr U.
\]
Thus $r_{\mathscr U}$ belongs to the ultrafilter limit
set of $X_{\mathrm{fin}}$ relative to $\mathcal S_{\mathrm{fin}}$.
By Lemma~\ref{spectral}, $X_{\mathrm{fin}}$ is spectral.
\end{proof}

\medskip
\begin{remark}\label{diff1}
Let $R$ be an integral domain. Denote by $\mathrm{SStar}(R)$ the set of all
semistar operations on $R$, and by $\mathrm{SStar}_f(R)$ the subset of finitary
semistar operations. Endow $\mathrm{SStar}(R)$ with the Zariski topology and
$\mathrm{SStar}_f(R)$ with the induced subspace topology, as in
Definition~\ref{defnzartop}. It is shown in \cite{Ca-Da14} that
$\mathrm{SStar}_f(R)$ is a spectral space, whereas $\mathrm{SStar}(R)$ is not
spectral in general; for an explicit counterexample, see \cite[Proposition 4.4]{da18}.
\end{remark}

\bigskip
\section{A characterization result}\label{characterizationn}

This section provides a characterization of when a subspace of the set of all overmonoids of $H$ is quasi-compact.  

\medskip
\begin{theorem}\label{main2}
Let $H$ be a monoid, and let $G$ be its quotient groupoid.  Let
$\Delta\subseteq \mathcal{R}(G| H)$ be a nonempty subset of the set of all overmonoids
of $H$. Then $\Delta$ is quasi-compact (with respect to the subspace topology) if
and only if $r_{\Delta}$ is finitary.
\end{theorem}

\begin{proof}
Suppose first that $\Delta$ is quasi-compact. Let $A\subseteq G$ be arbitrary and
let $x\in A_{r_{\Delta}}$.If $x=0$, then $x\in\emptyset_{r_\Delta}$, and
    $\emptyset$ is a finite subset of $A$.
    Thus we may assume that $x\neq0$. By definition of $r_{\Delta}$, we have
$x\in AS$ for every $S\in \Delta$.

For each $a\in A$, define
\[
W(a)
:=
\{S\in \Delta : x\in aS\}.
\]

If $a\neq 0$, then $W(a)=\Delta\cap U(a^{-1}x)$, whereas if $a=0$, then  $W(0)=\emptyset$.
Since $U(a^{-1}x)$ is open in the Zariski topology on $\mathcal{R}(G| H)$, each
$W(a)$ is open in the subspace topology on $\Delta$. Moreover, for every
$S\in \Delta$ there exists $a\in A$ such that $S\in W(a)$. Hence
$\{W(a)\}_{a\in A}$ is an open cover of $\Delta$.

By quasi-compactness, there exist $a_1,\ldots,a_n\in A$ such that
\[
\Delta=\bigcup_{i=1}^n W(a_i).
\]
Let $F:=\{a_1,\ldots,a_n\}\subseteq A$. Then for any $S\in \Delta$, we have
$x\in a_i S\subseteq FS$ for some $i\in [1,n]$. Consequently,
\[
x\in \bigcap_{S\in \Delta} (FS\cup\{0\}) = F_{r_{\Delta}}.
\]
Thus, every element $x\in A_{r_{\Delta}}$ belongs to $F_{r_{\Delta}}$ for some
finite subset $F\subseteq A$. Hence,
\[
A_{r_{\Delta}}
\subseteq
\bigcup_{F\in \mathcal{P}_{\mathrm{fin}}(A)} F_{r_{\Delta}}.
\]
The reverse inclusion follows immediately from the  Axiom \textbf{M2} of $r_{\Delta}$,
and therefore $r_{\Delta}$ is finitary.

Conversely, suppose that $r_{\Delta}$ is finitary and assume, for contradiction,
that $\Delta$ is not quasi-compact. By Lemma~\ref{alex}, there exists a family
$\{x_i\}_{i\in I}\subseteq G$ such that $\{\Delta\cap U(x_i)\}_{i\in I}$ is an open
cover of $\Delta$ with no finite subcover. Equivalently, for every $S\in \Delta$
there exists $i\in I$ such that $x_i\in S$, and for every finite subset
$J\subseteq I$ there exists $S_J\in \Delta$ such that $x_j\notin S_J$ for all
$j\in J$. Note that each $x_i\neq 0$, since $\Delta\cap U(0)=\Delta$. Define
\[
A:=\{x_i^{-1}: i\in I\}\subseteq G^{\bullet}.
\]
For any $S\in \Delta$, there exists $i\in I$ such that $x_i\in S$, and hence $1\in AS$. It follows that
\[
1\in \bigcap_{S\in \Delta} (AS\cup \{0\}) = A_{r_{\Delta}}.
\]

We claim that for every finite subset $F\subseteq A$, we have
$1\notin F_{r_{\Delta}}$. Indeed, let
$F=\{x_{j_1}^{-1},\ldots,x_{j_n}^{-1}\}$ for some finite
$J=\{j_1,\ldots,j_n\}\subseteq I$. By assumption, there exists $S_J\in \Delta$
such that $x_{j_k}\notin S_J$ for all $k\in [1,n]$, which implies
$1\notin FS_J$. Since $S_J\in \Delta$, this yields
\[
1\notin \bigcap_{S\in \Delta} (FS\cup\{0\}) = F_{r_{\Delta}}.
\]

Thus, $1\in A_{r_{\Delta}}$ but $1\notin F_{r_{\Delta}}$ for any finite
$F\subseteq A$, contradicting the assumption that $r_{\Delta}$ is finitary.
Therefore, $\Delta$ must be quasi-compact, completing the proof.
\end{proof}

For any nonempty $\Delta \subseteq \mathcal{R}(G | H)$, the generalized  $H$-module system $r_{\Delta}$ is idempotent. 
Indeed, for every $A \subseteq G$, we have
\[
A_{r_{\Delta}} 
= 
\bigcap_{S \in \Delta} (SA\cup\{0\}) 
\supseteq A,
\]
and hence $(A_{r_{\Delta}})_{r_{\Delta}} \supseteq A_{r_{\Delta}}.$
On the other hand,
\[
(A_{r_{\Delta}})_{r_{\Delta}}
=
\left( \bigcap_{S \in \Delta} (SA\cup \{0\}) \right)_{r_{\Delta}}
=
\bigcap_{T \in \Delta} T \left( \bigcap_{S \in \Delta} (SA\cup \{0\}) \right)
\subseteq \]

\[\subseteq\bigcap_{T \in \Delta} T(TA\cup\{0\})
=
A_{r_{\Delta}}.
\]
Therefore, $(A_{r_{\Delta}})_{r_{\Delta}} = A_{r_{\Delta}}$, and thus $r_{\Delta}$ is idempotent.

Consequently, Theorem~\ref{main2} also holds for $H$-module systems.
\bigskip

\begin{remark}\label{diff2}
Let $R$ be an integral domain with quotient field $K$, and let $\Delta \subseteq \mathcal{R}(K| R)$ be a collection of overrings of $R$. According to \cite[Corollary~2.8]{Ca-Da14}, if $\Delta$ is quasi-compact, then the semistar operation $\star_{\Delta}$ is finitary. However, the converse does not hold in general, as shown in \cite[Example~3.6]{Ca-Ma-Da15}.
\end{remark}

\medskip
\begin{lemma}[\cite{hoch69} Proposition 4]\label{Hochster} For a topological space $X$ the following conditions are equivalent.
\begin{enumerate}
    \item $X$ is a spectral space.
    \item  $X$ is quasi-compact, sober, and has a basis of quasi-compact open subspaces which is closed under finite intersections.
\end{enumerate}
\end{lemma}

\bigskip
We proceed by presenting two applications of Theorem \ref{main2}. 

\medskip
\begin{corollary}
Let
$
\Delta=\{H_{\mathfrak{p}} : \mathfrak{p}\in s\text{-}\mathrm{spec}(H)\}$
be the collection of localizations of $H$ at its $s$-prime ideals. Then
$r_{\Delta}$ is finitary.
\end{corollary}

\begin{proof} 
The map
\[
\eta:s\text{-spec}(H)\longrightarrow\Delta,
\qquad \mathfrak p\longmapsto H_{\mathfrak p},
\]
is bijective, since
\(\mathfrak p=H\setminus H_{\mathfrak p}^{\times}\).
For \(x\in G\) and \(h\in H^\bullet\), we have
\[
\eta^{-1}\bigl(\Delta\cap U(x)\bigr)
=\bigcup_{a\in(H:x)}D_s(a),
\qquad
\eta(D_s(h))=\Delta\cap U(h^{-1}).
\]
Thus $\eta$ is a homeomorphism. By Theorem~\ref{pronconst}, the space $s$-$\mathrm{spec}(H)$ is spectral, and hence
quasi-compact by Lemma~\ref{Hochster}. The conclusion now follows from
Theorem~\ref{main2}.
\end{proof}

\medskip
\begin{corollary}
Let
$
\Delta=\mathrm{Zar}(G| H)
$
be the set of all valuation overmonoids of $H$. Then $r_{\Delta}$ is finitary.
\end{corollary}

\begin{proof}
By Theorem~\ref{Zar}, the space $\Delta$ is spectral, and therefore
quasi-compact by Lemma~\ref{Hochster}. The claim follows again from
Theorem~\ref{main2}.
\end{proof}

\medskip
Next, we generalize \cite[Proposition 2.5]{Ca-Da14} and \cite[Proposition 2.7]{Ca-Da14} to module systems of monoids. 

\medskip
\begin{proposition}\label{prop1}
Let $H$ be a monoid with quotient groupoid $G$. Define the natural map
\[
\iota:\mathcal{R}(G| H)\longrightarrow \mathcal{X},
\qquad
S\longmapsto r_{\{S\}},
\]
where both spaces are endowed with the Zariski topology. Then $\iota$ is a
topological embedding.
\end{proposition}

\begin{proof}
Since $H_{r_{\{S\}}}=S$ for every $S\in \mathcal{R}(G| H)$, the map $\iota$ is
injective. To prove that $\iota$ is continuous, it suffices to show that the preimage of any
subbasic open set of $\mathcal{X}$ is open in $\mathcal{R}(G| H)$. Let
$A\subseteq G$. Then
\[
S\in \iota^{-1}(U_A)
\iff r_{\{S\}}\in U_A
\iff 1\in AS.
\]
Hence,
\[
\iota^{-1}(U_A)
=
\{S\in \mathcal{R}(G| H): 1\in AS\}
=
\bigcup_{a\in A^{\bullet}} U(a^{-1}),
\]
which is open in $\mathcal{R}(G| H)$. Thus $\iota$ is continuous.

Finally,  we show that for all $x\in G$, $\iota(U(x))$ is open in the image of $\iota$ endowed with the subspace topology. If $x=0$, then $\iota(U(0))=\iota(\mathcal{R}(G|H))$.
For  $x\in G^\bullet$, we claim that
\[
\iota(U(x))
=
U_{\{x^{-1}\}}\cap \iota(\mathcal{R}(G| H)).
\]
Indeed, if $r\in \iota(U(x))$, then $r=r_{\{S\}}$ for some overmonoid
$S\in \mathcal{R}(G| H)$ such that $x\in S$. Equivalently, $1\in x^{-1}S$,
which implies $r_{\{S\}}\in U_{\{x^{-1}\}}$, and hence
$r\in U_{\{x^{-1}\}}\cap \iota(\mathcal{R}(G| H))$.

Conversely, if $r\in U_{\{x^{-1}\}}\cap \iota(\mathcal{R}(G| H))$, then
$r=r_{\{S\}}$ for some $S\in \mathcal{R}(G| H)$ and $1\in x^{-1}S$, which
implies $x\in S$, and hence $r\in \iota(U(x))$.

Therefore, $\iota(U(x))$ is open in $\iota(\mathcal{R}(G| H))$ with the
subspace topology. This shows that $\iota$ is a topological embedding.
\end{proof}

\medskip
\begin{proposition}\label{prop2}
    Let $\tau\subseteq \mathcal{X}_{\textup{fin}}$ be a nonempty quasi-compact subspace.  Then $\bigwedge(\tau)$ is finitary. 
\end{proposition}
\begin{proof}
   Fix a subset  $A\subseteq G$ and an element $x\in A_{\bigwedge(\tau)}$. If \(x=0\), then \(x\in\emptyset_{\bigwedge(\tau)}\)
by Axiom \textbf{(Id1)}. Since \(\emptyset\) is a finite
subset of \(A\), the required finite witness already exists.
Thus, we may assume that \(x\neq 0\). By definition of $\bigwedge(\tau)$, we have  $x\in A_r$ for all $r\in \tau$, which implies that there exists a finite subset $E^{(r)}\subseteq A$ such that $x\in (E^{(r)})_r $ for all $r\in \tau$. Hence $1\in x^{-1}(E^{(r)})_r$, equivalently $r\in U_{x^{-1}E^{(r)}}$ for all $r\in \tau$. Thus $$\{U_{x^{-1}E^{(r)}}:r\in \tau\}$$ is an open cover of $\tau$. By quasi-compactness of $\tau$, it admits a finite subcover $$\{U_{x^{-1}E^{(r_i)}}:i\in [1,n]\}.$$ Set $E:=\bigcup_{i=1}^{n}E^{(r_i)}$, which is a finite subset of $A$. Now for any $r\in \tau$ there exists $U_{x^{-1}E^{(r_j)}}$ such that $r\in U_{x^{-1}E^{(r_j)}}$, by definition,  $x\in (E^{(r_j)})_r\subseteq E_r$ from which it follows that $x\in E_{\bigwedge(\tau)}$. Hence $\bigwedge(\tau)$ is finitary. 
\end{proof}

\bigskip
\noindent
\textbf{Conflict of interest statement.} The authors state that there is no conflict of interest.

\providecommand{\bysame}{\leavevmode\hbox to3em{\hrulefill}\thinspace}
\providecommand{\MR}{\relax\ifhmode\unskip\space\fi MR }
\providecommand{\MRhref}[2]{%
  \href{http://www.ams.org/mathscinet-getitem?mr=#1}{#2}
}



\begin{thebibliography}{10}



  

\bibitem{bag-fon04}    
S. ~El Baghdadi, M.~Fontana,  \emph{Semistar linkedness and flatness, Prüfer semistar multiplication domains},  Commun. Algebra \textbf{32}, (2004), pp.~1101 -- 1126.

\bibitem{dickman19}    
M. ~Dickmann, N.~Schwartz, M.~Tressl,  \emph{Spectral spaces}  Vol. 35 of New Mathematical Monographs, Cambridge University Press, Cambridge,  (2019).

\bibitem{dobb87}    
D. E. ~Dobbs, R.~Fedder, M.~Fontana,  \emph{Abstract Riemann surfaces of integral domains and spectral spaces},  Ann. Mat. Pura Appl. \textbf{148}, (1987), pp.~101 -- 115.

\bibitem{dobb86}    
D. E. ~Dobbs,  M.~Fontana,  \emph{Kronecker Function Rings and Abstract Riemann
Surfaces},  J. Algebra \textbf{99}, (1986), pp.~263 -- 274.

\bibitem{Ca-Da14}  
C. A. ~Finocchiaro, D.~Spirito,  \emph{Some topological considerations on semistar operations},  {J}ournal of {A}lgebra \textbf{409}, (2014), pp.~199 -- 218.

\bibitem{Ca14}  
C. A. ~Finocchiaro,   \emph{Spectral spaces and ultrafilters},  Commun. Algebra \textbf{42} , (2014), pp.~1496 -- 1508.

\bibitem{Ca-Ma-Da15}
C.~A. Finocchiaro, M.~Fontana, D.~Spirito,
\textit{New distinguished classes of spectral spaces: A survey},
in: S.~Chapman, M.~Fontana, A.~Geroldinger, B.~Olberding (eds.),
\textit{Multiplicative Ideal Theory and Factorization Theory:
Commutative and Non-commutative Perspectives},
Springer Proceedings in Mathematics \& Statistics, vol.~170,
Springer, Cham, 2016, pp.~117--143.

\bibitem{fin-fon-lop13}
C. A. ~Finocchiaro, M.~ Fontana,   K. A. ~Loper, \emph{The constructible topology on spaces of valuation domains}, Trans. Amer. Math. Soc \textbf{365}  (2013), pp. ~6199--6216.

\bibitem{fin-fon-dar16}
C. A. ~Finocchiaro, M.~ Fontana,   D. ~Spirito, \emph{A topological version of Hilbert's Nullstellensatz}, J. Algebra \textbf{461}  (2016), pp. ~25--41.

\bibitem{fo-lo-forum} M. ~Fontana, K.A.~ Loper, \emph{A generalization of Kronecker function rings and Nagata rings},
Forum Math. \textbf{19} (2007), no. 6, 971–1004.

\bibitem{fo-lo-survey} M.~Fontana, K.~A. Loper,
\textit{An historical overview of Kronecker function rings,
Nagata rings, and related star and semistar operations},
in: J.~W. Brewer, S.~Glaz, W.~J. Heinzer, B.~M. Olberding (eds.),
\textit{Multiplicative Ideal Theory in Commutative Algebra:
A Tribute to the Work of Robert Gilmer},
Springer, New York, 2006, pp.~169--187.

\bibitem{fo-lo-jpaa-2009} M. ~Fontana, K.A.~ Loper,  \emph{Cancellation properties in ideal systems: a classification of e.a.b. semistar operations}, 
J. Pure Appl. Algebra \textbf{213} (2009), no. 11, 2095–2103. 

\bibitem{fon-lop03} 
M. ~Fontana, K.A.~ Loper,   \emph{Nagata rings, Kronecker function rings, and related semistar operations}, Commun. Algebra \textbf{31}  (2003), pp. ~4775--4805.

\bibitem{fon-jar-san04}
M. ~Fontana, P.~ Jara, E. ~Santos,   \emph{Local-global properties for semistar operations}, Commun. Algebra \textbf{32}  (2004), pp. ~3111--3137.



\bibitem{alfred-kim-loper}
A. ~Geroldinger,  H. ~Kim, K. A. ~Loper,  \emph{On Long-Term Problems in Multiplicative Ideal Theory and Factorization Theory}, \url{https://arxiv.org/abs/2502.21020} (2025).




  \bibitem{h-c98} F.~Halter-Koch, \emph{Ideal Systems; An Introduction to Multiplicative Ideal Theory}, Marcel Dekker, Inc, (1998).
  
  \bibitem{halter-koch-2025} F.~Halter-Koch, \emph{Ideal theory of commutative rings and monoids},
  Lecture Notes in Mathematics 2368,
  Springer, Cham, 2025.

  \bibitem{h-c03} F.~Halter-Koch, \emph{Weak module systems and applications: a multiplicative theory of integral elements and the Marot property}, In: Commutative Ring Theory and Applications, Lecture Notes in Pure and Appl. Math., vol. 231, Marcel Dekker (2003), pp. ~213--231.

   \bibitem{h-c11} F.~Halter-Koch, \emph{Multiplicative ideal theory in the context of commutative monoids}, In: M. Fontana, S.-E. Kabbaj, B. Olberding, I. Swanson (eds.)  Commutative Algebra: Noetherian and Non-noetherian Perspectives, Springer (2011), pp. ~203--231.

   \bibitem{h-c01} F.~Halter-Koch, \emph{Localizing systems, module systems and semistar operations}, J. Algebra \textbf{238}  (2001), pp. ~723--761.

\bibitem{hoch69} M.~Hochster, \emph{Prime ideal structure in commutative rings}, Trans. Amer. Math. Soc. \textbf{142}  (1969), pp. ~43--60.
   

\bibitem{jech98} T.~Jech, \emph{Set Theory}, Springer, New York, (1997).

     \bibitem{ju12} J.~Juett, \emph{Generalized comaximal factorization of ideals}, J. Algebra \textbf{352}  (2012), pp. ~141--166.

\bibitem{Krull-35} W. Krull, \emph{Idealtheorie}, Springer Berlin, Heidelberg, 1935. 

\bibitem{Ok-Ma94}
A. ~Okabe, R.~Matsuda,  \emph{Semistar operations on integral domains},  Toyama Math. J. \textbf{17} (1994), pp.~1 -- 21.

  
      \bibitem{da18} D.~Spirito, \emph{The Zariski topology on the sets of semistar operations without finite-type assumptions}, J. Algebra \textbf{513}  (2018), pp. ~27--49.

      \bibitem{doniyor26} D.~Yazdonov, \emph{On the structure of length sets with maximal elasticity}, Commun. Algebra \textbf{54(5)}  (2026), pp. ~2145--2158.

       \bibitem{zar44} O.~Zariski, \emph{The compactness of the Riemann manifold of an abstract field of algebraic
functions}, Bull. Amer. Math. Soc. \textbf{50}  (1944), pp. ~683--691.

  \bibitem{zar-sam75} O.~Zariski, P.~Samuel, \emph{Commutative Algebra},  Springer Verlag, Graduate Texts in Mathematics \textbf{29}, vol. 2, New York, (1975).


\end{thebibliography}

\end{document}